\documentclass[11pt]{amsart}
\usepackage{amsmath, amsthm, amssymb}

\newtheorem{theorem}{Theorem}

\newtheorem{remark}{Remark}  
\newtheorem{definition}{Definition}  
\newtheorem{lemma}{Lemma}  
\newtheorem{proposition}{Proposition}
\newtheorem{example}{Example}  

\def\loc{\hbox{\normalfont\scriptsize loc}}
\def\L{L}
\def\R{{\mathbb R}}
\def\eps{\varepsilon}

\title[Threshold solutions]{Threshold and strong threshold solutions \\
     of a semilinear parabolic equation}
\author{Pavol Quittner}
\address{Department of Applied Mathematics and Statistics, Comenius University
         Mlynsk\'a dolina, 84248 Bratislava, Slovakia}
\email{quittner@fmph.uniba.sk}

\begin{document}

\begin{abstract}
If $p>1+2/n$ then 
the equation $u_t-\Delta u = u^p, \quad x\in\R^n,\ t>0,$
possesses both positive global solutions
and positive solutions which blow up in finite time.
We study the large time behavior of radial positive solutions lying 
on the borderline between global existence and blow-up.
\end{abstract}

\maketitle
%------------------------------------------------------------------
\section{Introduction} \label{sec-intro}

In this paper we consider positive classical 
solutions of the Cauchy problem
\begin{equation} \label{Fuj}
\left\{\quad\begin{aligned}
u_t-\Delta u &= u^p &\qquad& x\in\R^n,\ t>0, \\
u(x,0) &= u_0(x), &\qquad& x\in\R^n,
\end{aligned}\right.
\end{equation}
with $p>p_F:=1+2/n$ and $u_0\in BC^+\setminus\{0\}$, where
$BC$ denotes the space of bounded continuous 
functions in $\R^n$ and $BC^+:=\{\phi\in BC:\phi\geq0\}$.
We will study the large-time behavior of solutions
lying on the borderline between global existence and blow-up. 
Positive solutions of \eqref{Fuj}
which blow up in finite time exist for all $p>1$. 
For such $p$,
the assumption $p>p_F$ is necessary and sufficient for the existence
of positive global solutions.

In addition to the exponent $p_F$ we will often work with
the following critical exponents:
$$ \begin{aligned}
 p_{sg} &:=\begin{cases}\ +\infty       & \hbox{ if }n\leq2, \\
                       1+\frac2{n-2} & \hbox{ if }n\geq3,
           \end{cases} \qquad
 p_S :=\begin{cases}\ +\infty       & \hbox{ if }n\leq2, \\
                       1+\frac4{n-2} & \hbox{ if }n\geq3,
           \end{cases} \\
 p_{JL} &:=\begin{cases}\ +\infty & \hbox{ if }n\leq10, \\
         1+4\frac{n-4+2\sqrt{n-1}}{(n-2)(n-10)} & \hbox{ if }n>10.
          \end{cases}
\end{aligned}
$$
If $p>p_{sg}$ then the function
\begin{equation} \label{Ustar}
u_*(x):={\L}|x|^{-2/(p-1)}, \quad\hbox{where}\quad
{\L}^{p-1}:=\frac2{(p-1)^2}\bigl((n-2)p-n\bigr),
\end{equation}
is a singular steady state of \eqref{Fuj}.
Notice that $L=L(n,p)\to0$ as $p\to p_{sg}+$;
we set $L=L(n,p):=0$ if $p\leq p_{sg}$.
If $p\geq p_S$ then for each $\alpha>0$ there exists a unique positive 
radially symmetric and radially decreasing steady state $u_\alpha$
 of \eqref{Fuj} satisfying $u_\alpha(0)=\alpha$. These solutions
intersect each other (and the singular steady state $u_*$)
if and only if $p<p_{JL}$.
It was shown in \cite{GNW92} that the following is true:

\begin{proposition} \label{prop-GNW}
\strut
\hbox{\rm(i)} Let $p\in[p_S,p_{JL})$ and $\alpha>0$.
If $u_0\geq u_\alpha$ and $u_0\not\equiv u_\alpha$ then
the solution of \eqref{Fuj} blows up in finite time.
If $u_0\leq u_\alpha$ and $u_0\not\equiv u_\alpha$ then
the solution of \eqref{Fuj} exists globally and
$\|u(\cdot,t)\|_\infty\to0$ as $t\to\infty$.

\strut\hbox{\rm(ii)} Let $p\geq p_{JL}$ and $\alpha>0$.
If $u_0\geq \lambda u_\alpha$ for some $\lambda>1$ then
the solution of \eqref{Fuj} blows up in finite time.
If $u_0\leq \lambda u_\alpha$ for some $\lambda\in(0,1)$ then
the solution of \eqref{Fuj} exists globally and
$\|u(\cdot,t)\|_\infty\to0$ as $t\to\infty$.
On the other hand, the solution $u_\alpha$ is stable
in a suitable weighted Lebesgue space, and, in particular,
the properties in \hbox{\rm(i)} do not hold. 
\end{proposition}

Inspired by this result we say that a solution $u^*$ of \eqref{Fuj}
with initial data $u_0^*$
is a {\bf threshold solution} if the following is true:
If $u_0\geq\lambda u_0^*$ for some $\lambda>1$
then the solution of \eqref{Fuj} blows up in finite time;
if $u_0\leq\lambda u_0^*$ for some $\lambda<1$
then the solution of \eqref{Fuj} exists globally.
We say that a solution $u^*$ of \eqref{Fuj}
with initial data $u_0^*$
is a {\bf strong threshold solution} if the following holds:
If $u_0\geq u_0^*$, $u_0\not\equiv u_0^*$ then 
the solution of \eqref{Fuj} blows up in finite time;
if $u_0\leq u_0^*$, $u_0\not\equiv u_0^*$ then
the solution of \eqref{Fuj} exists globally.
If $u^*$ is a threshold solution but not a strong threshold
solution then we say that it is a {\bf weak threshold solution}.
Hence, in particular, the steady states $u_\alpha$ are strong or weak
threshold solutions if $p\in[p_s,p_{JL})$ or $p\geq p_{JL}$,
respectively.

Threshold solutions have been intensively studied.
It is well known that their behavior strongly depends
on the exponent $p$ and the spatial decay of the initial data;
we recall some of the corresponding results below. 
On the other hand, the question, whether the threshold
solution is weak or strong, has attracted much less attention.
Even the fact that for any $p>p_F$ there exists
both weak and strong threshold solutions
(which will be an easy consequence of our results)
does not seem to be known.

The existence of weak threshold solutions for any $p>p_F$
follows from the following theorem.

\begin{theorem} \label{thm-weak}
Given $p>p_F$ there exists $L^*=L^*(n,p)>0$ with the following
properties:
Let $u_0\in BC^+$ and let $u$ denote the solution of \eqref{Fuj}.

\noindent\strut\hbox{\rm(i)}
If $u_0(x)|x|^{2/(p-1)}\leq L^*$ for all $x$ then $u$ is global.\\
\strut\kern5mm
In addition, if $p<p_{JL}$ then
$\|u(\cdot,t)\|_\infty\leq Ct^{-1/(p-1)}$ as $t\to\infty$.

\noindent\strut\hbox{\rm(ii)}
If $\liminf_{|x|\to\infty}u_0(x)|x|^{2/(p-1)}>L^*$
then $u$ blows up in finite time.

\noindent\strut\hbox{\rm(iii)}
$L^*(n,p)>L(n,p)$ if $p<p_{JL}$, $L^*(n,p)=L(n,p)$ if $p\geq p_{JL}$.
\end{theorem}

In fact, taking $u_0$ such that 
$$ u_0(x)|x|^{2/(p-1)}\leq L^* \ \hbox{ for all }\ x
 \quad\hbox{and}\quad
 \lim_{|x|\to\infty}u_0(x)|x|^{2/(p-1)}=L^*,$$ 
Theorem~\ref{thm-weak}
guarantees that $u$ is a weak threshold solution.
In addition, this solution decays to zero as $t\to\infty$ if $p<p_{JL}$.
On the other hand, it is known that this solution may decay to zero
(i.e.~$\lim\|u(\cdot,t)\|_\infty\to0$ as $t\to\infty$),
converge to a positive steady state, grow up
(i.e.~$\lim\|u(\cdot,t)\|_\infty\to \infty$ as $t\to\infty$),
or exhibit a more complicated behavior
if $p\geq p_{JL}$, see \cite{GNW01,GNW92,PY03}.

Theorem~\ref{thm-weak} follows from
\cite[Theorem 1.3 and Remark 1.3]{Nai12} if $p<p_{JL}$ or 
\cite[Theorem 20.5]{SPP} and \cite[Theorem 4(iii)]{GNW01}
if $p\geq p_{JL}$.   
In Section~\ref{sec5} we prove Theorem~\ref{thm-weak} for $p<p_{JL}$
by using different arguments than those in \cite{Nai12}.

The existence and behavior of strong threshold solutions 
are much less understood. If $p>p_S$ then \cite[Lemma 2]{PY14} shows
that any radially symmetric threshold solution with $u_0\in BC^+\cap H^1(\R^n)$
is a strong threshold solution, 
and it is also known that such solutions blow up in finite time,
see \cite{Mi02,Mi05,MM11,Sou16}.
On the other hand, the only known global strong threshold
solutions seem to be the steady states $u_\alpha$ for $p\in[p_S,p_{JL})$.
In order to study the existence and properties of strong threshold
solutions we will restrict ourselves to the initial data 
$u_0\in X$, where
\begin{equation} \label{X}
\left\{\quad\begin{aligned}
&\hbox{$X$ is the set of nonnegative, continuous, radially symmetric}\\
&\qquad\hbox{and radially nonincreasing functions in $\R^n$.}
\end{aligned}\right.
\end{equation}
The main technical tool in our study will be the following
characterization of strong threshold solutions.

\begin{theorem} \label{thm-strong}
Let $\|\cdot\|$ denote the norm in $L^\infty(\R^n)\cap L^1(\R^n)$.
If $u_0\in X$  then the corresponding solution $u$ of \eqref{Fuj}
is a strong threshold solution
if and only if for each $\eps>0$ there exist  
$v_1,v_2\in X$ such that
$\|u_0-v_1\|+\|u_0-v_2\|<\eps$, and
the solution of \eqref{Fuj} with initial data $v_1$ or $v_2$
exists globally or blows up in finite time, respectively.
\end{theorem}

Theorem~\ref{thm-strong} guarantees, in particular,
that any threshold solution with initial data 
$u_0\in X\cap L^1(\R^n)$ 
is a strong threshold solution.
However, we will mainly use this theorem 
to study strong threshold solutions
with slow spatial decay, 
satisfying $\lim_{|x|\to\infty}u_0(x)|x|^{2/(p-1)}=L^*$,
for example.
We will show that
(in addition to the above mentioned strong threshold solutions
which blow up in finite time if $p>p_S$,
and the strong threshold steady states for $p_S\leq p<p_{JL}$), 
for any $p>p_F$ there exist global strong threshold solutions 
(GSTS for short) with initial data in $X$. In particular, 

\begin{itemize}
\item
if $p_F<p<p_{JL}$ then there exists a GSTS $u$ which decays to zero,

\item if $p\geq p_{JL}$ then there exists a GSTS $u$
which grows up,

\item if $p=p_S$ then there exists a non-stationary
 GSTS $u$ which converges to a positive steady state,

\item if $p=p_S$ and $n=3$ then there exists a GSTS $u$
which grows up,

\item if $p>p_S$ then there exists a GSTS $u$ such that
\begin{equation} \label{osc0} 
0=\liminf_{t\to\infty}\|u(\cdot,t)\|_\infty<
  \limsup_{t\to\infty}\|u(\cdot,t)\|_\infty=\infty.
\end{equation}
\end{itemize}

Global solutions with large-time behaviors just mentioned
have mostly been known (see \cite{PY14} in the case of \eqref{osc0} and
$p_S<p<p_{JL}$, for example), but it is not clear whether
those solutions are strong threshold solutions.
On the other hand, we often use those solutions or the methods of 
proofs of their existence in order to prove 
the existence of a GSTS with the same large-time behavior.  

It is known that if $p\geq p_{JL}$ then the singular steady state $u_*$
is a strong threshold singular solution in a suitable sense (see \cite{GV97}).
The following theorem shows that
an analogous result is true in the case $p<p_{JL}$.

\begin{theorem} \label{thm-strong2}
Assume $p_F<p<p_{JL}$ and let $L^*=L^*(n,p)$ be the constant defined
in Theorem~\ref{thm-weak}.
Let $\tilde u$ be the minimal weak solution of \eqref{Fuj}
with singular initial data $L^*|x|^{-2/(p-1)}$.
Then $\tilde u$ is a global self-similar solution
of the form $\tilde u(x,t)=t^{-1/(p-1)}w(|x|/\sqrt t)$ with $w$ bounded,
and the solution with initial data $\tilde u(\cdot,t_0)$ 
is a strong threshold solution for any $t_0>0$. 
\end{theorem}

We summarize known (and our) results on the behavior of weak/strong 
threshold solutions with initial data in $X$ in Table~\ref{tab-ws}:
If a cell contains just ``YES'', for example,
then this means that there exist both weak and strong
threshold solutions with the designated property;
in the row ``Convergence to a positive steady state''
we only consider {\it non-stationary\/} solutions.
The results in Table~\ref{tab-ws} for $p<p_S$ follow from
\cite{Kaw96,PQS2} and  Theorems~\ref{thm-weak}, \ref{thm-strong},
global existence of all threshold solutions if $p=p_S$ 
or all weak threshold solutions if $p>p_S$ (i.e.~``NO'' for blow-up)
follows from \cite{GV97} or \cite[the proof of Lemma~2]{PY14},
respectively,
and the remaining results have already been mentioned above.
Some results related to the question marks 
for strong threshold solutions can be found in
Proposition~\ref{prop-GB} and Remark~\ref{rem-pS}(ii). 
If $p=p_S$ then \cite[Conjecture 1.1]{FK12} suggests that
a growing-up strong threshold solution  should exist if $n\leq4$.

%-----------------------------------------------------------------
\dimen1=40mm \dimen2=23mm \dimen3=17mm \dimen4=25mm \dimen5=19mm
\newcount\colcount 
\def\skipcol{\kern \ifcase\colcount 0mm\or \dimen2\or \dimen3\or \dimen4\or
\dimen5\fi%
      \advance\colcount by1}
\def\hsz#1{\dimen255=\ifcase \colcount \dimen1\or \dimen2\or \dimen3\or
\dimen4\or \dimen5\fi%
        \ifnum #1>1 \advance\dimen255 by% 
              \ifcase \colcount \relax\or \dimen3\or \dimen4\or
\dimen5\fi\fi%
        \ifnum #1>2 \advance\dimen255 by%
               \ifcase\colcount \relax\or \dimen4\or \dimen5\fi\fi%
        \ifnum #1>3 \advance\dimen255 by \dimen5\fi}
\def\vrv#1{\ifcase #1\relax\or 
  \vrule height6mm depth4mm\or  
  \vrule height11mm depth9mm\or 
  \vrule height16mm depth14mm\or
  \vrule height21mm depth19mm\or 
  \vrule height26mm depth24mm\or 
  \vrule height31mm depth29mm\fi}
\def\Cell#1#2#3{\vbox to6mm{\hsz{#1}\hbox
to\dimen255{\hss{#3}\hss\vrv#2}\hrule\vss}%
      \advance\colcount by#1}
\def\newrow#1{\colcount=0\hbox{\vrv1#1}}
%----------------------------------------------------------

\begin{table}
\centerline{$$
\vtop{\offinterlineskip%
\hrule%  
\newrow{%
\Cell11{\relax}%
\Cell11{$p_F<p<p_S$}%
\Cell11{$p=p_S$}%
\Cell11{$p_S<p<p_{JL}$}%  
\Cell11{$p\geq p_{JL}$}}% 
\newrow{%
\Cell11{Decay to zero}%
\Cell31{YES}% 
\Cell11{YES/?}}%
\newrow{%
\Cell11{Steady states}%
\Cell15{NO}%
\Cell21{NO/YES}% 
\Cell11{YES/NO}}%
\newrow{%
\Cell11{\vbox to6mm{\vss\kern4mm\hbox{\ Convergence to a}\hbox{positive
steady state}\vss}}%
\skipcol
\Cell11{?/YES}%
\Cell12{?}%
\Cell11{YES/?}}%
\newrow{%
\Cell11{Grow-up}%
\skipcol
\Cell11{?/YES${}^{*}$}% 
\skipcol
\Cell11{YES}}%
\newrow{%
\Cell11{Blow-up}%
\skipcol
\Cell11{NO}% 
\Cell21{NO/YES}}%  
\newrow{%
\Cell11{Other}%
\skipcol
\Cell11{?}%  
\Cell11{?/YES}%
\Cell11{YES}}% 
}%
$$}
\kern4mm
   \caption[Table]{%
Possible behavior of weak/strong threshold \\
\strut\kern18.5mm solutions  of \eqref{Fuj}  with initial data in $X$.\\   
YES${}^{*}$\dots if $n=3$; see Remark~\ref{rem-pS}(i)\\
}
   \label{tab-ws}
\end{table}

%------------------------------------------------------------

It should be mentioned that the corresponding notions 
of threshold and strong threshold solutions 
coincide in the case of the Cauchy-Dirichlet problem
\begin{equation} \label{FujOmega}
\left\{\quad\begin{aligned}
u_t-\Delta u &= u^p, &\qquad& x\in\Omega,\ t>0, \\
u &= 0, &\qquad& x\in\partial\Omega,\ t>0, \\
u(x,0) &= u_0(x), &\qquad& x\in\Omega,
\end{aligned}\right.
\end{equation}
where $\Omega$ is a smooth bounded domain in $\R^n$
and $u_0\in C^1(\overline\Omega)$, $u_0\geq0$,
$u_0=0$ on $\partial\Omega$.
In addition, the behavior of threshold solutions
of \eqref{FujOmega} is well understood if
$\Omega$ is a ball and $u_0$ is radially symmetric
and radially nonincreasing (see Table~\ref{tab2}),
and the results remain true in a more general situation.
In fact, let $\Omega\subset\R^n$ be bounded, smooth, 
and let $u_0\in L^\infty(\Omega)$, $u_0\geq0$.
%
%---------------------------------------------------------
\dimen4=17mm
\begin{table}
\centerline{$$
\vtop{\offinterlineskip%
\hrule%
\newrow{%
\Cell11{\relax}
\Cell11{$1<p<p_S$}%
\Cell11{$p=p_S$}%
\Cell11{$p>p_S$}}%
\newrow{%
\Cell11{Steady states}%
\Cell12{YES}%
\Cell22{NO}}%
\newrow{%
\Cell11{\vbox to6mm{\vss\kern4mm\hbox{\ Convergence to a}\hbox{positive steady state}\vss}}%
}%
\newrow{%
\Cell11{Grow-up}%
\Cell13{NO}%
\Cell11{YES}%
\Cell11{NO}}%
\newrow{%
\Cell11{Blow-up}%
\skipcol
\Cell11{NO}%
\Cell11{YES}}%
\newrow{%
\Cell11{Other}%
\skipcol
\Cell21{NO}}%
}
$$}
\kern4mm
   \caption[Table]{%
Possible behavior of threshold solutions of \eqref{FujOmega}
with $\Omega=\{x\in\R^n:|x|<R\}$ and $u_0\in C(\overline\Omega)$ being
radially symmetric and radially nonincreasing, $u_0=0$ on $\partial\Omega$.}
   \label{tab2}
\end{table}
%---------------------------------------------------------
%
If $p<p_S$ then any threshold solution of \eqref{FujOmega}
is global, bounded, and its $\omega$-limit set consists
of nontrivial equilibria.
If $p>p_S$ and $\Omega$ is convex then the threshold
solutions blow up in finite time due to
\cite{CDZ07} (see also \cite[Proposition 6.6(ii)]{BS15}).
On the other hand, if $p\geq p_S$ and $\Omega$ is an annulus, for example,
then any radial threshold solution converges to 
the unique positive radial steady state:
This follows from the estimates in the proof of 
\cite[Theorem 4.1]{PQS2} and \cite[Theorem 1.2]{NN85}.

Let us also mention that the behavior of non-radial
threshold solutions of the Cauchy problem \eqref{Fuj} is open even
in the subcritical case: If $p_F<p<p_S$  then the 
global existence and decay of 
such solutions is only known under the additional assumption
$p<n(n+2)/(n-1)^2$ or $n\leq2$, see \cite{PQS2,Q16},
or for exponentially decaying initial data,
see \cite{Kaw96} or \cite[Theorem~28.9]{SPP};
global existence (without decay) is also known
for $u_0\in H^1(\R^n)\cap L^\infty(\R^n)\cap L^{(p+1)/p}(\R^n)$,
see \cite[Theorem~1.2(v)]{Q03}. 
On the other hand, if $p>p_S$ and the (non-radial) initial data $u_0\in BC^+$
are continuously differentiable and satisfy either
$$ |\nabla u_0|^2+u_0^{p+1}\in L^q(\R^n)\hbox{ for some }
q\in\Bigl[1,\frac n2\frac{p-1}{p+1}\Bigr),$$
or
$$ u_0(x)+|x||\nabla u_0(x)|=o(|x|^{-2/(p-1)})\hbox{ as }|x|\to\infty,$$
then the threshold solution blows up in finite time, see \cite{Sou16}.

%=======================================================
\section{Preliminaries} \label{sec-prel}

We will consider classical solutions of \eqref{Fuj}
of the form $u(x,t)=U(|x|,t)$.
In particular, $u_*(x,t)=U_*(|x|,t)$ and $u_\alpha(x,t)=U_\alpha(x,t)$,
where $u_*$ and $u_\alpha$ are the singular and regular steady states 
defined in the Section~\ref{sec-intro}.

By $z$ we denote the zero number functional on the interval $[0,\infty)$.
More precisely, given $\varphi\in C([0,\infty))$, 
we set $z(\varphi)=0$ if $\varphi$ does not change sign, and
$$\left\{\quad\begin{aligned}
 z(\varphi):=\sup\{k:\exists \ \ & 0\leq r_1< r_2<\dots<r_{k+1}\ \hbox{ such that } \\
     &\varphi(r_i)\varphi(r_{i+1})<0 \ \hbox{ for }\ i=1,2,\dots,k\}
\end{aligned}\right.
$$
otherwise. Similarly, $z_{[0,R]}$ denotes the zero number functional on the
interval $[0,R]$.
Recall also that $X$ is defined in \eqref{X}, set
\begin{equation} \label{X1}
 X_1:=\{\Phi\in C([0,\infty)):\hbox{$\Phi$ is nonincreasing, $\Phi\geq0$}\}
\end{equation}
and notice that $u_0\in X$ 
if and only if $u_0(x)=\Phi(|x|)$ for some $\Phi\in X_1$.

The proof of the following lemma is obvious.

\begin{lemma} \label{lem-thr}
Assume $u_0(x)=\Phi(|x|)$ for some $\Phi\in X_1$.
Then the corresponding solution of \eqref{Fuj}
is a threshold solution if and only if for each $\eps>0$ there
exist $\Phi_\eps^+,\Phi_\eps^-\in X_1$ such that
\begin{equation} \label{thresh1}
\left\{\ \begin{aligned}
& 0\leq\Phi_\eps^-\leq\Phi\leq\Phi_\eps^+,\\
& \hbox{the solution with initial data $\Phi_\eps^-(|x|)$
   exists globally}, \\
& \hbox{the solution with initial data $\Phi_\eps^+(|x|)$
  blows up in finite time},
\end{aligned}\right.
\end{equation} 
and 
$$\Phi_\eps^+-\Phi_\eps^-\leq\eps\Phi.$$
\end{lemma}

\begin{definition} \label{deftf} \rm
Assume $u_0(x)=\Phi(|x|)$ for some $\Phi\in X_1$.
We say that the corresponding solution $u$ of \eqref{Fuj}
is an {\bf $L^1$-threshold solution} if for each $\eps>0$ there
exist $\Phi_\eps^+,\Phi_\eps^-\in X_1$ such that
\eqref{thresh1} is true and
$$ 
 \|\Phi_\eps^+-\Phi_\eps^-\|_\infty+
  \int_0^\infty|\Phi_\eps^+(r)-\Phi_\eps^-(r)|r^{n-1}\,dr \leq\eps, 
$$
\end{definition}

%---------------------------------------------------------
\begin{lemma} \label{lemma1}
Assume $p_S\leq p<p_{JL}$. 
Let $u$ be a threshold or an $L^1$-threshold
solution with initial data $u_0\in X$
and maximal existence time $T$,
and let $\Phi,\Phi_\eps^\pm$ be as in Lemma~\ref{lem-thr}
or Definition~\ref{deftf}, respectively.
Fix $\beta>0$ and assume $z(\Phi-U_\beta)=1$.
If $\Phi(0)<U_\beta(0)$ and
$z(\Phi_\eps^- -U_\beta)=1$ for each $\eps$ small 
then $u(0,t)<U_\beta(0)$ for all $t\in(0,T)$.  
If  $\Phi(0)>U_\beta(0)$,  and
$z(\Phi_\eps^+ -U_\beta)=1$ for each $\eps$ small 
then $u(0,t)>U_\beta(0)$ for all $t\in(0,T)$. 
\end{lemma}

\begin{proof}
We have $u(x,t)=U(|x|,t)$ with $U(r,0)=\Phi(r)$.

First assume $\Phi(0)<U_\beta(0)$ and
$z(\Phi_\eps^- -U_\beta)=1$ for each $\eps$ small.
Assume on the contrary that $U(0,t_0)\geq U_\beta(0)$ for some $t_0\in(0,T)$.
Then the zero number properties  (see \cite[Proposition 2.6]{PY14s})
and the maximum principle guarantee
$z(U(\cdot,t)-U_\beta)=0$ and $U(\cdot,t)>U_\beta$ for $t>t_0$.
Fixing such $t$, the continuous dependence on initial values guarantees
the existence of $\eps>0$ small such that 
the solution $v_\eps^-(x,t)=V_\eps^-(|x|,t)$ of \eqref{Fuj}
with initial data $V_\eps^-(r,0)=\Phi_\eps^-$ satisfies
$V_\eps^-(0,t)>U_\beta(0)$, hence 
$z(V_\eps^-(\cdot,t)-U_\beta)=0$ and $V_\eps^-(\cdot,t)>U_\beta$.
Now  Proposition~\ref{prop-GNW} 
contradicts the global existence of $v_\eps^-$.

Next assume $\Phi(0)>U_\beta(0)$ and
$z(\Phi_\eps^+ -U_\beta)=1$ for each $\eps$ small.
Assume on the contrary that $U(0,t_0)\leq U_\beta(0)$ for some
$t_0\in(0,T)$. 
Let $v_\eps^+(x,t)=V_\eps^+(|x|,t)$ denote the solution of \eqref{Fuj}
with initial data $V_\eps^+(r,0)=\Phi_\eps^+$
Then similarly as above we obtain the existence of $\eps>0$
and $t>t_0$ such that $V_\eps^+(\cdot,t)<U_\beta$.
Since $v_\eps^+$ blows up in finite time,
Proposition~\ref{prop-GNW} yields a contradiction.
\end{proof}

%---------------------------------------------------------
\begin{lemma} \label{lemma2}
Let $v(x,t)=V(|x|,t)$ be a global bounded positive solution of \eqref{Fuj}
satisfying $z(V_t(\cdot,t_0))<\infty$ for some $t_0>0$.
Then $V(0,t)\to\beta\in[0,\infty)$ as $t\to\infty$.
If $\beta>0$ then $V(\cdot,t)\to U_\beta$ in $C_{\loc}$ as $t\to\infty$. 
\end{lemma}

\begin{proof}
Since the function $t\mapsto z(V_t(\cdot,t))$ is finite 
and nonincreasing for $t>t_0$, and it drops whenever $V_t(0,t)=0$
(see \cite[Proposition 2.6]{PY14s}),
there exists $t_1>t_0$ such that $V_t(0,t)$ does not change sign for
$t>t_1$.
Therefore, $V(0,t)\to\beta\in[0,\infty)$ as $t\to\infty$.
Next assume $\beta>0$.
Then any $v_\infty$ in the $\omega$-limit set (in $C_{\loc}$)
of $v$ is a bounded entire
radial nonnegative solution of \eqref{Fuj}
satisfying $v_\infty(0,t)=\beta$ for all $t\in\R$,
$v_\infty(x,t)=V_\infty(|x|,t)$.
If 
$$
\hbox{$V_\infty(r_0,t_0+\tau_0)\ne V_\infty(r_0,t_0)$ 
for some $r_0>0$, $t_0\in\R$ and $\tau_0>0$}
$$
then $V_\infty(r_0,t+\tau_0)\ne V_\infty(r_0,t)$ for $t\in[t_0,t_0+\eps)$,
hence
$z_{[0,r_0]}(V_\infty(\cdot,t+\tau_0)-V_\infty(\cdot,t))$
is finite for $t\in(t_0,t_0+\eps)$ and drops at any $t\in(t_0,t_0+\eps)$,
which yields a contradiction, cf.~\cite[Corollary 2.9]{MM04}.
Consequently, $V_\infty(\cdot,t)=U_\beta$ for each $t$.
\end{proof}

%=======================================================
\section{Proof of Theorem~\ref{thm-strong}} \label{sec-strong}

In this section we will 
prove Theorem~\ref{thm-strong}
by showing that each $L^1$-threshold solution
is a strong threshold solution.

%------------------------------------------------------
\begin{proposition} \label{prop2}
Let $\phi,\psi\in C([0,\infty))$ be nonnegative,
$z(\phi-\psi)=1$, $\phi(0)>\psi(0)$,
$\int_0^\infty \psi(r)r^{n-1}\,dr < \int_0^\infty \phi(r)r^{n-1}\,dr < \infty$.
Let $u_0\in X$, $u_0>0$.

(i) 
Assume that $u_0(x)+\phi(|x|)$ and $u_0(x)+\psi(|x|)$ are radially nonincreasing.
Let $u_\phi$ or $u_\psi$ be the solution of \eqref{Fuj} with initial data
$u_0(x)+\phi(|x|)$ or $u_0(x)+\psi(|x|)$, respectively.
Fix $M>u_\phi(0,0)$ and assume that $u_\psi(0,t)<M$ for $t<t_0$ and
$u_\psi(0,t_0)=M$ for some $t_0>0$.
Then $u_\phi(0,t)=M$ for some $t\leq t_0$.
 
(ii)
Assume that $u_0(x)-\phi(|x|)$ and $u_0(x)-\psi(|x|)$ are nonnegative and
radially nonincreasing,
Let $u_\phi$ or $u_\psi$ be the solution of \eqref{Fuj} with initial data
$u_0(x)-\phi(|x|)$ or $u_0(x)-\psi(|x|)$, respectively.
If $u_\psi(0,t)\leq M$ for some $M>0$ and $t\leq t_0$,
then $u_\phi(0,t)\leq M$ for $t\leq t_0$.
\end{proposition}

\begin{proof}
We will only prove (i); the proof of (ii) is analogous.
As above, we will write $u_\phi,u_\psi$ in the form
$u_\phi(x,t)=U_\phi(|x|,t)$ and 
$u_\psi(x,t)=U_\psi(|x|,t)$.

Assume on the contrary $u_\phi(0,t)<M$ for $t\leq t_0$. 
We have 
$$z(U_\phi(\cdot,0)-U_\psi(\cdot,0))=z(\phi-\psi)=1$$
and $U_\phi(0,0)>U_\psi(0,0)$, $U_\phi(0,t_0)<U_\psi(0,t_0)$.
The zero number properties guarantee the existence of
$t_1\in(0,t_0)$ such that
$z(U_\phi(\cdot,t)-U_\psi(\cdot,t))=1$ for $t<t_1$,
$z(U_\phi(\cdot,t)-U_\psi(\cdot,t))=0$ for $t>t_1$,
hence
$U_\phi\leq U_\psi$ for $t\geq t_1$.
Consequently, denoting
$$f(t):=\int_{\R^n}(u_\phi(x,t)-u_\psi(x,t))\,dx,$$
we have $f(0)>0$ and $f(t_1)\leq0$. 

Given $t\in(0,t_1)$, let $r_t$ be the unique zero of
$U_\phi(\cdot,t)-U_\psi(\cdot,t)$ , i.e.
$U_\phi(r_t,t)=U_\psi(r_t,t)$.
Set $c_t:=pU_\phi(r_t,t)^{p-1}$.
Since $u_\phi,u_\psi$ are radially nonincreasing, we have
$$(U_\phi^p-U_\psi^p)(r,t)=pU_\theta^{p-1}(r,t)(U_\phi-U_\psi)(r,t) 
\geq c_t(U_\phi-U_\psi)(r,t), $$ 
where we used 
$U_\theta(r,t)\geq U_\psi(r,t)\geq U_\psi(r_t,t)=U_\phi(r_t,t)$ if $r\leq r_t$
and, similarly, $U_\theta(r,t)\leq U_\phi(r_t,t)$ if $r\geq r_t$.
We also have
\begin{equation} \label{eq2}
\left\{\quad\begin{aligned}
\frac{\partial}{\partial t}(u_\phi-u_\psi)
&=\Delta(u_\phi-u_\psi)+(u_\phi^p-u_\psi^p) \\
&=\Delta(u_\phi-u_\psi)+a(x,t)(u_\phi-u_\psi) \\
&\geq \Delta(u_\phi-u_\psi)+c_t(u_\phi-u_\psi), \\
\end{aligned}\right.
\end{equation}
where $a(x,t):=pu_\theta^{p-1}(x,t)\in[0,pM^{p-1}]$ 
and $(U_\phi-U_\psi)(\cdot,0)=\phi-\psi$.
Integrating \eqref{eq2} over $\R^n$  
we obtain
$$f'(t)=\int_{\R^n} a(x,t)(u_\phi-u_\psi)(x,t)\,dx\geq c_t f(t),$$ 
hence $f(t)\geq f(0)>0$ for $t\leq t_1$,
which yields a contradiction.
\end{proof}

%------------------------------------------------------
\begin{remark} \label{rem-prop} \rm
Assumption 
$\phi(0)>\psi(0)$ in Proposition~\ref{prop2}
can be replaced by the following assumption:
There exist $r\geq0$ such that
$\phi(r)>\psi(r)$ and $\phi\geq\psi$ on $[0,r]$.
In fact, this more general assumption
guarantees that the more restrictive assumption
will be satisfied after a time shift.
\qed
\end{remark} 

%------------------------------------------------------
\begin{lemma} \label{lem-t0}
Let $u$ be an $L^1$-threshold solution with initial data $u_0(x)=\Phi(|x|)$
and maximal existence time $T$.
Fix $t_0\in(0,T)$ and let $v$ be the solution of \eqref{Fuj}
with initial data $u(\cdot,t_0)$.
Then $v$ is an $L^1$-threshold solution.
\end{lemma}

\begin{proof}
Let $\Phi_\eps^\pm$ be from Definition~\ref{deftf},
$\tilde u_0:=u(\cdot,t_0)$.
If $\eps$ is small enough then the continuous dependence
of solutions of \eqref{Fuj} on initial data in $L^\infty$
guarantees that the solutions $u_\eps^\pm$ with initial data
$\Phi_\eps^\pm$ exist for $t\leq t_0$, and there
exists $\tilde\eps=\tilde\eps(\eps,u_0,t_0)>0$ such that
$\|u_\eps^\pm(\cdot,t_0)-\tilde u_0\|_\infty\leq\tilde\eps/2$
and $\tilde\eps\to0$ as $\eps\to0$.
The function $w:=u_\eps^\pm-u$ does not change sign,
belongs to $L^1(\R^n)$ for fixed $t$ and solves the equation 
$w_t=\Delta w+aw$, where $a=a(x,t)$ is bounded for $t\leq t_0$
(see the proof of Proposition~\ref{prop2}).
Integrating this equation and enlarging $\tilde\eps$
if necessary we obtain
$\|u_\eps^\pm(\cdot,t_0)-\tilde u_0\|_{L^1(\R^n)}\leq\tilde\eps/2$.
Consequently, the functions
$\tilde\Phi(|x|):=\tilde u_0(x)$ 
and
$\tilde\Phi_{\tilde\eps}^\pm(|x|):=\tilde u_\eps^\pm(x,t_0)$ 
satisfy the conditions in Definition~\ref{deftf}
(with $u_0$ and $\eps$ replaced by $\tilde u_0$ and $\tilde\eps$).
 \end{proof}

%------------------------------------------------------
\begin{lemma} \label{lemma3}
Let $u$ be an $L^1$-threshold solution with initial data $u_0(x)=\Phi(|x|)$.
Then, for any $\delta>0$ small enough, 
there exist 
$\phi_1,\phi_2\in X_1$ such that
$\delta\geq\phi_i(0)>0=\phi_i(\delta)$, $i=1,2$, and the following is true:

(i) The solution
of \eqref{Fuj} with initial data
$u_0(x)+\phi_1(|x|)$  blows up in finite time.

(ii) If $\Phi(r)<\Phi(0)$ for $r>0$ then 
$\Phi-\phi_2\in X_1$
and the solution of \eqref{Fuj} with initial data 
$u_0(x)-\phi_2(|x|)$ exists globally.
\end{lemma}

\begin{proof}
Let $\Phi_\eps^\pm$ be from Definition~\ref{deftf}.

(i) Fix $\delta>0$ small and $\phi\in X_1$ satisfying
$\delta\geq\phi(0)>0=\phi(\delta)$. 
Consider $\eps>0$ small
and choose $\varphi_\eps\in X_1$ such that 
$\eps=\varphi_\eps(0)=\varphi_\eps(\delta)>\varphi_\eps(2\delta)=0$.
Set $\tilde\Phi_\eps:=\max(\Phi_\eps^+,\Phi+\varphi_\eps)$
and notice that $\tilde\Phi_\eps=\Phi+\eps$ on $[0,\delta]$.
If $\eps$ is small enough then the functions $\phi$ and $\psi:=\tilde\Phi_\eps-\Phi$
satisfy the assumptions in Proposition~\ref{prop2}(i).
Since the solution $u_\psi$ with initial data
$u_0(x)+\psi(|x|)=\tilde\Phi_\eps(|x|)\geq\Phi_\eps^+(|x|)$
blows up in finite time, 
Proposition~\ref{prop2}(i) 
and the radial monotonicity of $u_\psi$ guarantee that the solution of \eqref{Fuj}
with initial data $u_0(x)+\phi(|x|)$ blows up in finite time.

(ii) 
Assume that $\Phi(r)<\Phi(0)$ for $r>0$.
Choose $\delta>0$ small and $r_0\in(0,\delta)$ such that 
$\Phi(0)-\Phi(r_0)<\delta$.
Set $\phi(r):=\max(0,\Phi(r)-\Phi(r_0))$. 
If $\eps>0$ is small enough then the functions $\phi$ and $\psi:=\Phi-\Phi_\eps^-$
satisfy either $\psi\leq\phi$ or
the assumptions in Proposition~\ref{prop2}(ii).
Since the solution with initial data $u_0(x)-\psi(|x|)=\Phi_\eps^-(|x|)$
exists globally, the comparison principle or
Proposition~\ref{prop2}(ii) and the radial monotonicity
guarantee that the solution of \eqref{Fuj}
with initial data $u_0(x)-\phi(|x|)$ exists globally.
\end{proof}

%------------------------------------------------------
\begin{lemma} \label{lemma4}  
Let $u$ be an $L^1$-threshold solution,
$u_0(x)=\Phi(|x|)$.
Let $\psi:[0,\infty)\to[0,\infty)$ be bounded and continuous,
$\psi\not\equiv0$. Then the solution of \eqref{Fuj}
with initial data $u_0(x)+\psi(|x|)$ blows up in finite time.
If $u_0(x)\geq\psi(|x|)$ then the solution of \eqref{Fuj}
with initial data $u_0(x)-\psi(|x|)$ exits globally.
\end{lemma}

\begin{proof}
Choose $r_1>0$ such that $\psi(r_1)>0$ 
and fix a smooth function $\tilde\psi:[0,\infty)\to[0,\infty)$
with support in $[r_1/2,2r_1]$ such that $\tilde\psi\leq\psi$
and $\tilde\psi\not\equiv0$.
Fix $t_0>0$ small and notice that the solution
$\tilde u(x,t)=\tilde U(|x|,t)$ with initial data $u_0(x)+\tilde\psi(|x|)$
satisfies $\tilde u(0,t_0)>u(0,t_0)$ by the strong comparison principle.
Consider $\delta\in(0,r_1/2)$ 
and let $\phi_1$ be the function from Lemma~\ref{lemma3}.
The continuous dependence on initial data guarantees that
choosing $\delta$ small enough,
the solution $u_{\phi_1}(x,t)=U_{\phi_1}(|x|,t)$ with initial data
$u_0(x)+\phi_1(|x|)$ satisfies
$U_{\phi_1}(0,t_0)<\tilde U(0,t_0)$.
Since $U_{\phi_1}(0,0)>\tilde U(0,0)$
and $z(U_{\phi_1}(\cdot,0)-\tilde U(\cdot,0))=1$, 
we have $\tilde U(\cdot,t)\geq U_{\phi_1}(\cdot,t)$ for $t\geq t_0$.
Since $u_{\phi_1}$ blows up in finite time,
$\tilde u$ blows up as well,
and, consequently, the solution with initial data
$u_0(x)+\psi(|x|)$ also blows up.

The proof of global existence for the initial data
$u_0(x)-\psi(|x|)$ is analogous if $\Phi(r)<\Phi(0)$ for $r>0$.
If this assumption fails then one can choose $t_0>0$ small and
replace $u_0$ and $\psi$ with
$\tilde u_0:=u(\cdot,u_0)$ and $\tilde\psi:=\tilde u_0-u_\psi(\cdot,t_0)$
where $u_\psi$ is the solution with initial data $u_0(x)-\psi(|x|)$.
Then $\tilde u_0$ radially decreasing, and it is
an $L^1$-threshold solution due to Lemma~\ref{lem-t0}
so that one use the above arguments to show the global existence of
$u_\psi$.
\end{proof}
 
%------------------------------------------------------
\begin{theorem} \label{thm2}
Any $L^1$-threshold solution is a strong threshold solution.
\end{theorem}

\begin{proof}
Let $u$ be an $L^1$-threshold solution with initial data 
$u_0(x):=\Phi(|x|)$.
Let $\varphi:\R^n\to[0,\infty)$ be bounded and continuous, $\varphi\not\equiv0$.
We want to prove that the solution $u_\varphi$ of \eqref{Fuj} with initial
data $u_0+\varphi$ blows up in finite time.
Fix $\tau>0$ small and set $\tilde u_0(x):=u(x,\tau)$.
Since $u_\varphi(\cdot,\tau)>\tilde u_0$, there exists 
a bounded continuous function 
$\psi:[0,\infty)\to[0,\infty)$ such that
$\psi\not\equiv0$ and $u_\varphi(x,\tau)\geq \tilde u_0(x)+\psi(|x|)$.
Now the comparison principle and Lemmas~\ref{lem-t0} and~\ref{lemma4}
guarantee that $u_\varphi$ blows up in finite time.

If $0\leq\psi\leq u_0$ 
then the global existence of the solution
with initial data $u_0-\psi$ follows analogously. 
\end{proof}

%------------------------------------------------------
\begin{proof}[Proof of Theorem~\ref{thm-strong}]
If $u_0$ is a strong threshold solution
and $\eps>0$ then any functions $v_1,v_2\in X$
satisfying $v_1\leq u_0\leq v_2$, $v_i\not\equiv u_0$, $i=1,2$,
and $\|u_0-v_1\|+\|u_0-v_2\|<\eps$
fulfil the conditions in Theorem~\ref{thm-strong}.

On the other hand, if $\eps>0$ and $v_1,v_2$
are the functions in Theorem~\ref{thm-strong}
then $\Phi_\eps^-(|x|):=\min(u_0(x),v_1(x))$
and $\Phi_\eps^+(|x|):=\min(u_0(x),v_2(x))$
fulfil the conditions in Definition~\ref{deftf},
hence $u$ is an $L^1$-threshold solution.
Now Theorem~\ref{thm2} guarantees that $u$ is
a strong threshold solution.
\end{proof}

%=======================================================
\section{Examples of strong threshold solutions} \label{sec-exam}

In this section we provide examples of various global strong threshold solutions.
In addition, we also discuss possible behavior of strong threshold solutions
with initial data $\Phi(|x|)$ satisfying $z(\Phi-U_*)<\infty$.

\begin{example} \rm \label{exam-decay}
Let $p\in[p_S,p_{JL})$. Let $L^*$ be from Theorem~\ref{thm-weak},
$W(r):=L^*r^{-2/(p-1)}$, and let $r_1$ be defined by 
$W(r_1)=1$.
Consider initial data $\Phi_\alpha=\Phi_\alpha(|x|)$ of the form
\begin{equation} \label{ini-decay}
 \Phi_\alpha(r):=\begin{cases}
1 &\hbox{ if }r\leq r_1+\alpha,\\
\max(m+(r-r_1-\alpha)W'(r_1),W(r)) &\hbox{ if }r>r_1+\alpha,
\end{cases}
\end{equation}
where $\alpha\geq0$, see Figure~\ref{fig-decay}.

\def\red{\relax}
%---------------------------------------------
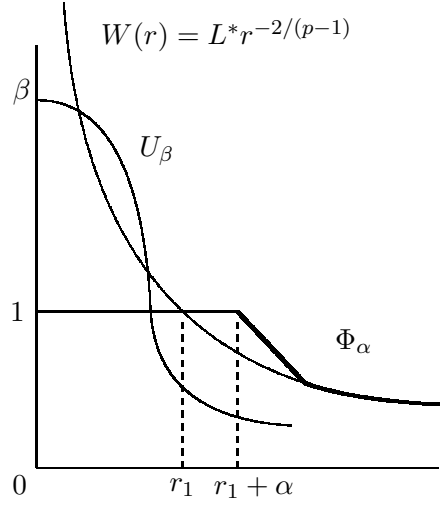
\begin{figure}[ht]
\centering
\begin{picture}(240,200)(0,0)
\unitlength=0.8pt
\put(12,12){\makebox(0,0)[c]{$0$}}
\put(110,225){\makebox(0,0)[c]{$W(r)=L^*r^{-2/(p-1)}$}}
\put(77,170){\makebox(0,0)[c]{$U_\beta$}}
\put(13,198){\makebox(0,0)[c]{$\beta$}}
\put(11,94){\makebox(0,0)[c]{$1$}}
\put(170,80)%
{\makebox(0,0)[c]{$\red\Phi_{\alpha}$}}
\put(89,11){\makebox(0,0)[c]{$r_1$}}
\put(105,11){\makebox(0,0)[c]{\strut\kern10mm$r_1+\alpha$}}
\put(20,20){\line(1,0){190}}
\put(20,20){\line(0,1){200}}
{\linethickness{1pt}\red
\put(20,94){\line(1,0){95}}
\bezier500(115,94)(131,77)(147,60)
}
\put(89,20){\line(0,1){3}}
\put(89,26){\line(0,1){3}}
\put(89,32){\line(0,1){3}}
\put(89,38){\line(0,1){3}}
\put(89,44){\line(0,1){3}}
\put(89,50){\line(0,1){3}}
\put(89,56){\line(0,1){3}}
\put(89,62){\line(0,1){3}}
\put(89,68){\line(0,1){3}}
\put(89,74){\line(0,1){3}}
\put(89,80){\line(0,1){3}}
\put(89,86){\line(0,1){3}}
\put(115,20){\line(0,1){3}}
\put(115,26){\line(0,1){3}}
\put(115,32){\line(0,1){3}}
\put(115,38){\line(0,1){3}}
\put(115,44){\line(0,1){3}}
\put(115,50){\line(0,1){3}}
\put(115,56){\line(0,1){3}}
\put(115,62){\line(0,1){3}}
\put(115,68){\line(0,1){3}}
\put(115,74){\line(0,1){3}}
\put(115,80){\line(0,1){3}}
\put(115,86){\line(0,1){3}}
\bezier500(33,240)(45,55)(210,50)  
 \bezier500(20,194)(69,194)(74,94) 
 \bezier500(74,94)(77,44)(140,40)%
{\linethickness{1pt}\red
\bezier500(147,60)(170,52)(210,50)  
}
\end{picture}
\kern5mm
\caption{Initial data $\Phi_{\alpha}$ of a time-decaying strong threshold
solution if $p\in[p_S,p_{JL})$.}
   \label{fig-decay}
\end{figure}
%--------------------------------------------------

Set also $\tilde\Phi_m(r):=\min(m,W(r))$ for $m>0$
and notice that $\Phi_0=\tilde\Phi_1$.
If $\alpha>0$ is small then the solution $u^{(\alpha)}$ with initial data
$\Phi_\alpha$
exists globally due to Proposition~\ref{prop2}
(with $u_0(x):=\Phi_0(|x|)$, $\phi:=\tilde\Phi_m-\Phi_0$, $m>1$,
$\psi:=\Phi_\alpha-\Phi_0$).
On the other hand, $u^{(\alpha)}$ blows up in finite time if $\alpha$ is
large enough.
Therefore, the threshold value
$\alpha^*:=\sup\{\alpha>0: u^{(\alpha)}\hbox{ is global}\}$
satisfies $\alpha^*\in(0,\infty)$
and Theorem~\ref{thm2} shows that 
the solution $u^*$ with initial data $\Phi_{\alpha^*}$
is a strong threshold solution.
Fix $\beta>1$ large enough such that the steady state $U_\beta$
satisfies $U_\beta(r)<W(r)$ for $r\geq r_1$.
Then $z(\Phi_\alpha-U_\beta)=1$ and Lemma~\ref{lemma1}
guarantees $u^*(0,t)<U_\beta(0)$ for all $t>0$,
hence $u^*$ is global and bounded.
Now Lemma~\ref{lemma2} shows that $u^*$ either
converges to a positive steady state or decays to zero.
Since $z(\Phi_{\alpha^*}-U_*)=1$
and $z(U_\gamma-U_*)>1$ for any $\gamma>0$,
we see that $u^*$ cannot converge to $u_\gamma$,
hence it decays to zero.  
\qed
\end{example}

\begin{example} \rm \label{exam-conv}
Let $p=p_S$.
Fix $\beta\in(0,1)$ and let $r_0$ be the (only)
zero of $U_1-U_\beta$.
Given $\alpha\in[0,1]$, set
$$
\Phi_\alpha:=\begin{cases}
 \frac12(U_1+U_\beta) &\hbox{ on }[r_0,\infty),\\
 \alpha U_1+(1-\alpha)U_\beta &\hbox{ on }[0,r_0),
\end{cases}
$$
and consider the solution 
$u^{(\alpha)}(x,t)=U^{(\alpha)}(|x|,t)$ with initial data $\Phi_\alpha(|x|)$.
Then $z(\Phi_\alpha-U_1)=z(\Phi_\alpha-U_\beta)=1$ and $z((U^{(\alpha)})_t(\cdot,t))<\infty$  
for $\alpha\in(0,1)$ and $t>0$.
Proposition~\ref{prop-GNW}
and the continuous dependence on initial data
guarantee that $u^{(\alpha)}$ blows up in finite time if $\alpha$ is close 1,
while $u^{(\alpha)}$ exists globally if $\alpha$ is close to 0.
Fix $\alpha^*=\sup\{\alpha:u^{(\alpha)}\hbox{ exists globally}\}$.
Then the corresponding threshold solution $u^{(\alpha^*)}$ exists   
globally and $U^{(\alpha^*)}(0,t)\in(U_\beta(0),U_1(0))$ for all $t$,
due to Lemma~\ref{lemma1} and the radial monotonicity of $U^{(\alpha^*)}$.
Since $z(\frac{d}{dt}U^{(\alpha^*)}(\cdot,t))<\infty$ for $t>0$, 
Lemma~\ref{lemma2} guarantees $u^{(\alpha^*)}(x,t)\to U_\gamma(|x|)$
as $t\to\infty$ for some $\gamma\in[\beta,1]$.

If $p=p_S$ and $n\geq7$ then
a more detailed description of the asymptotic behavior
of solutions converging to a positive steady state 
(and having initial data close to this steady state in the energy
space $\dot H^1$) can be found in \cite{CMR}.
\qed
\end{example}

%------------------------------------------------------
\begin{example} \rm \label{exam1}
Let $p\geq p_{JL}$.
We will construct global strong threshold solutions 
which grow up.

Fix $0<m<\tilde m$, let $R,\tilde R$
be defined by $U_*(R)=m$, $U_*(\tilde R)=\tilde m$,
and $\tilde\Phi(r):=\min(\tilde m,U_*(r))$.
Given $\alpha\geq0$, set also
\begin{equation} \label{ini-GU}
 \Phi_\alpha(r)=\Phi_{\alpha,m}(r):=\begin{cases}
m &\hbox{ if }r\leq R+\alpha,\\
\max(m+(r-R-\alpha)U_*'(R),U_*(r)) &\hbox{ if }r> R+\alpha,
\end{cases} 
\end{equation}
cf.~similar initial data in Example~\ref{exam-decay}.
Notice that $\Phi_\alpha\geq\Phi_0=\min(m,U_*)$,
the solutions with initial data $\Phi_0$ and $\tilde\Phi$ exist globally
and grow up as $t\to\infty$ due to \cite[Theorem 6.1]{PY03},
and $z(\Phi_\alpha-U_*)=1$ if $\alpha>0$.
If $\alpha>0$ is small then the solution $u^{(\alpha)}$ with initial data $\Phi_\alpha$
exists globally due to Proposition~\ref{prop2}
(with $u_0(x):=\Phi_0(|x|)$, $\phi:=\tilde\Phi-\Phi_0$, $\psi:=\Phi_\alpha-\Phi_0$).
On the other hand, $u^{(\alpha)}$ blows up in finite time if $\alpha$ is large enough. 
The comparison principle and \cite[Theorem~29.1]{SPP}
guarantee that
the strong threshold solution $u^{(\alpha^*)}$ either
grows up or blows up in finite time $T$.
If $p>p_{JL}$ then \cite[Theorem 5.28]{MM09} guarantees
that $u^{(\alpha^*)}$ grows up
(since blow-up of a threshold solution cannot be complete).

The above arguments do not guarantee the existence of a strong threshold
solutions with grow-up for $p=p_{JL}$ and, in addition, 
they are based on a nontrivial result from \cite{MM09}. 
In order to overcome these drawbacks
let us provide a more complicated construction
which --- in some sense --- combines the above arguments
and suitable modifications of the ideas in 
\cite{PY14} (cf.~also  Example~\ref{exam3} below).

Fix $m_1\in(0,1)$, set $\Phi^{(1)}_\alpha:=\Phi_{\alpha,m_1}$,
where $\Phi_{\alpha,m}$ is defined by \eqref{ini-GU},
and let $\alpha_1^*$ be the threshold value.
If the strong threshold solution with initial data $\Phi^{(1)}_{\alpha_1^*}$
exists globally then it grows up and we are done.
If this solution blows up in finite time 
then choose 
$\alpha_1\in(0,\alpha_1^*)$ such that 
$v_1(x):=(\Phi^{(1)}_{\alpha_1^*}-\Phi^{(1)}_{\alpha_1})(|x|)$ satisfies
$\|v_1\|<1$, where
$\|\cdot\|$ denotes the norm in $L^\infty\cap L^1(\R^n)$.
Fix also $T_1:=1$ and $\tilde m_1\in(0,m_1)$ such that
$\Phi^{(1)}_{\alpha_1}(r)=U_*(r)$ whenever $\Phi^{(1)}_{\alpha_1}(r)<\tilde m_1$.
Let $u^{(1)}$ be the (global) solution with initial data
$\Phi^{(1)}_{\alpha_1}$.
The continous dependence on initial data 
guarantees the existence of $m_2\in(0,\min(\tilde m_1,1/2))$
such that the solution $u$ of \eqref{Fuj} exists on $[0,T_1]$
whenever $\|u_0-u^{(1)}(\cdot,0)\|_\infty\leq m_2$.

Next consider initial data of the form
$\Phi^{(2)}_\alpha:=\max(\Phi^{(1)}_{\alpha_1},\Phi_{\alpha,m_2})$ where
$\alpha>0$.
Proposition~\ref{prop2} and Remark~\ref{rem-prop} 
(used with
$u_0(x):=\Phi^{(1)}_{\alpha_1}(|x|)$, 
$\phi:=\Phi^{(1)}_{(\alpha_1^*+\alpha_1)/2}-\Phi^{(1)}_{\alpha_1}$,
$\psi:=\Phi^{(2)}_\alpha-\Phi^{(1)}_{\alpha_1}$)
guarantee that the corresponding solution
exists globally if $\alpha$ is small
and, obviously,
the solution blows up if $\alpha$ is large.
Let $\alpha_2^*$ be the threshold value.
If the strong threshold solution with initial data $\Phi^{(2)}_{\alpha_2^*}$    
exists globally then it grows up and we are done.
If this solution blows up in finite time then choose
$\alpha_2\in(0,\alpha_2^*)$ such that 
$v_2(x):=(\Phi^{(2)}_{\alpha_2^*}-\Phi^{(2)}_{\alpha_2})(|x|)$ satisfies
$\|v_2\|<1/2$.    
Fix also $T_2:=2$ and $\tilde m_2\in(0,m_2)$ such that
$\Phi^{(2)}_{\alpha_2}(r)=U_*(r)$ whenever $\Phi^{(2)}_{\alpha_2}(r)<\tilde m_2$.    
Let $u^{(2)}$ be the (global) solution with initial data
$\Phi^{(2)}_{\alpha_2}$.    
The continous dependence on initial data 
guarantees the existence of $m_3\in(0,\min(\tilde m_2,1/3))$
such that the solution $u$ of \eqref{Fuj} exists on $[0,T_2]$
whenever $\|u_0-u^{(2)}(\cdot,0)\|_\infty\leq m_3$.

Next consider initial data of the form
$\Phi^{(3)}_\alpha:=\max(\Phi^{(2)}_{\alpha_2},\Phi_{\alpha,m_3})$ where
$\alpha>0$, and proceed as above.
If the process does not stop at any final step then we find
sequences $T_k=k$, $m_1>m_2>\dots0$, $m_k<1/k$, 
a nondecreasing sequence of global solutions $u^{(k)}$,
functions $v_k\geq0$ satisfying $\|v_k\|<1/k$,
$u^{(k)}(\cdot,0)+v_k\in X$.
In addition,
the solution of \eqref{Fuj} exists on $[0,T_k]$
whenever $0\leq u_0\leq u^{(k)}(\cdot,0)+m_{k+1}$,
we also have
$u^{(j)}(\cdot,0)\leq u^{(k)}(\cdot,0)+m_{k+1}$ for all $j>k$,
and the solutions with initial data $u^{(k)}(\cdot,0)+v_k$
blow up in finite time.
It is now easy to see that the solution $u^*$ with initial data
$u^*_0:=\lim_{k\to\infty}u^{(k)}(\cdot,0)$
is a global strong threshold solution which grows up.
In fact, the estimate 
$u^*_0\leq u^{(k)}(\cdot,0)+m_{k+1}$ for each $k$
guarantees that $u^*$ is global.
Since the solutions with initial data
$\max(u^*(\cdot,0),u^{(k)}(\cdot,0)+v_k)\in X$ blow up in finite time
for each $k$, Theorem~\ref{thm-strong} guarantees that
$u^*$ is a strong threshold solution.
The comparison principle and \cite[Theorem~29.1]{SPP}
guarantee that this solution grows up.
\qed
\end{example}

%------------------------------------------------------
Assume $p\geq p_{JL}$.
If $u$ is a threshold solution with compactly supprted  initial data in $X$ 
then $u$ is a strong threshold solution due to
Theorem~\ref{thm2}, and it is also known that such solution blows up.
Example~\ref{exam1} shows that
a strong threshold solution of \eqref{Fuj}
with initial data in $X$
can grow up, 
and Example~\ref{exam3} will show that it can also satisfy \eqref{osc0}.
In the following proposition we show that
if a strong threshold solution with initial data $u_0\in X$ 
neither blows up
nor grows up then $u_0$ has to intersect the singular
steady state infinitely many times.

\begin{proposition}  \label{prop-GB}
Assume $p\geq p_{JL}$. Let $u$ be a strong threshold solution
with initial data $u_0\in X$, and let  $u_0(x)=\Phi(|x|)$
where $z(\Phi-U_*)<\infty$. Then either $u$ 
grows up or blows up in finite time.
\end{proposition}

\begin{proof}
Since $z(\Phi-U_*)<\infty$ there exists $r_0>0$ such that
either $\Phi(r)\geq U_*(r)$ for all $r\geq r_0$
or $\Phi(r)\leq U_*(r)$ for all $r\geq r_0$.
In the former case, the conclusion follows from
the comparison principle and \cite[Theorem 29.1]{SPP}.
In the latter case, 
fix $m_1\in(0,\min(1,\Phi(0),U_*(r_0)))$
and let $r_1>r_0$ be defined by $U_*(r_1)=m_1$. 
Example~\ref{exam1}(iii) guarantees the existence
of a growing-up strong threshold solution $u^G$
with radially nonincreasing initial data $u^G(x,0)=\Phi^G(|x|)$
satisfying $\Phi^G(r)=m_1$ for $r\leq r_1$,
$\Phi^G(r)\geq U_*(r)$ for $r>r_1$,
$\Phi^G\not\equiv U_*$ on $(r_1,\infty)$.
Since both $\Phi,\Phi^G$ are initial data
of strong threshold solutions, $\Phi^G\geq U_*\geq\Phi$
on $[r_1,\infty)$, $\Phi^G\equiv m_1<\Phi(0)$ on $[0,r_1]$
and $\Phi$ is nonincreasing on $[0,r_1]$, we have 
$\Phi(0)>\Phi^G(0)$ and
$z(\Phi-\Phi^G)=1$.
Fix $\delta>0$ such that $\Phi>\Phi^G$ on $[0,\delta]$,
and fix also
$\varphi\in X_1\setminus\{0\}$ with support in $[0,\delta]$.
Then the solution $u^{(\alpha)}$ with initial data
$(\Phi+\alpha\varphi)(|x|)$ blows up in finite time for any $\alpha>0$,
and $z(\Phi+\alpha\varphi-\Phi_{\alpha^*})=1$.
The same arguments as in Lemma~\ref{lemma1} imply
\begin{equation} \label{uuG0}
\hbox{$u(0,t)>u^G(0,t)$ for all $t<T$,}
\end{equation} 
where $T$ denotes the maximal existence time of $u$.
In fact, assume on the contrary $u(0,t_0)\leq u^G(0,t_0)$ for some $t_0<T$.
Then $u<u^G$ for $t>t_0$ by the zero number properties.
Fixing $t_1\in(t_0,T)$, there exists $\alpha>0$ small
such that $u^{(\alpha)}(0,t_1)<u^G(0,t_1)$,
hence $u^{(\alpha)}<u^G$ for $t\geq t_1$.
Since $u^{(\alpha)}$ blows up and $u^G$ is global,
this yields a contradiction.

Since $u^G(0,t)\to\infty$ as $t\to\infty$, 
the conclusion follows from \eqref{uuG0}.
\end{proof}

%------------------------------------------------
\begin{example} \rm \label{exam3}
Assume $p>p_S$. We will use modifications of the arguments in \cite{PY14}
in order to find strong threshold solutions satisfying \eqref{osc0}.

Given $\eps,R>0$, set 
$$\varphi_{\eps,R}(r):=
\begin{cases}
\eps &\hbox{ if }r\in[0,R],\\
\eps(R+1-r) &\hbox{ if }r\in(R,R+1),\\
0 &\hbox{ if }r\geq R+1. 
\end{cases}
$$
Scaling and comparison with solutions 
of the Cauchy-Dirichlet problem in a ball
show that, fixing $\eps>0$, 
the solution $u=u_{\eps,R}$
of \eqref{Fuj} with initial data 
$u_0(x)=\varphi_{\eps,R}(|x|)$
blows up in finite time
if $R$ is large enough.
In addition, there exists $\eps_1\in(0,1)$ small such that
the solution $u_1^R:=u_{\eps_1,R}$
exists globally if $R$ is small enough.
Set $R^*_1:=\sup\{R>0:u_1^R \hbox{ exists globally}\}$.
Then $u_1^{R^*_1}$ is a strong threshold solution with compact support, hence 
it blows up in a finite time $T^*_1$, see \cite[Theorem 5.15]{MM09},
for example. 
Fix $T_1\in(0,T^*_1)$ such that $\|u_1^{R^*_1}(\cdot,T_1)\|_\infty>1$
and
denote by $\|\cdot\|$ the norm in $L^\infty\cap L^1(\R^n)$.
Fix $R_1\in(0,R^*_1)$ such that 
the solution $u_1^{R_1}$ satisfies $\|u_1^{R_1}(\cdot,T_1)\|_\infty>1$
and the function 
$v_1:=(u_1^{R^*_1}-u_1^{R_1})(\cdot,0)$
satisfies $\|v_1\|<1$. 
The solution $u^{(1)}:=u_1^{R_1}$ with initial data
$u^{(1)}_0(x)=\varphi_{\eps_1,R_1}(|x|)$
exists globally and decays to zero (see \cite[Theorem 5.15]{MM09}),
hence there exists $\tilde T_1>\max(T_1,1)$ such that $\|u^{(1)}(\cdot,\tilde T_1)\|_\infty<1$.
The continuous dependence on initial data in $L^\infty$ guarantees the existence
of $\eps_2\in(0,\min(\eps_1,1/2))$ such that the solution of \eqref{Fuj} exists on $[0,\tilde T_1]$
and satisfies $\|u(\cdot,\tilde T_1)\|_\infty<1$ whenever
its initial data $u_0$ satisfy $0\leq u_0\leq u^{(1)}(\cdot,0)+\eps_2$.

Next consider solutions $u_2^R$  with initial data of the form 
$$u_2^R(x,0)=\max(u^{(1)}(x,0),\varphi_{\eps_2,R}(|x|)),\quad R>R_1+\eps_1-\eps_2.$$
Since these solutions exists globally if $R$ is close to $R_1+\eps_1-\eps_2$
and blow up in finite time if $R$ is large,
there exists a threshold value $R^*_2>R_1+\eps_1-\eps_2$.
The corresponding strong threshold solution $u_2^{R^*_2}$
blows up in a finite time $T^*_2>\tilde T_1$.
As above, we find $R_2\in(R_1,R^*_2)$ and $T_2\in(\tilde T_1,T^*_2)$ 
such that 
the solution $u_2^{R_2}$ satisfies $\|u_2^{R_2}(\cdot,T_2)\|_\infty>2$
and the function $v_2:=(u_2^{R^*_2}-u_2^{R_2})(\cdot,0)$
satisfies $\|v_2\|<1/2$.
The solution $u^{(2)}:=u_2^{R_2}$ exists globally and decays to zero,
hence there exists $\tilde T_2>\max(T_2,2)$ such that $\|u^{(2)}(\cdot,\tilde T_2)\|_\infty<1/2$.
The continuous dependence on initial data in $L^\infty$ guarantees the existence
of $\eps_3\in(0,\min(\eps_2,1/3))$ such that the solution of \eqref{Fuj} exists on $[0,\tilde T_2]$
and satisfies $\|u(\cdot,\tilde T_2)\|_\infty<1/2$ whenever
its initial data $u_0$ satisfy $0\leq u_0\leq u^{(2)}(\cdot,0)+\eps_3$.
 
Next we consider solutions with initial data
$$u_3^R(x,0)=\max(u^{(2)}(x,0),\varphi_{\eps_3,R}(|x|)),\quad R>R_2+\eps_2-\eps_3,$$
and proceed as above.
By induction we find a nondecreasing sequence of global solutions
$u^{(k)}$ and sequences $\eps_1>\eps_2>\dots$,
$0<T_1<\tilde T_1<T_2<\tilde T_2<T_3<\dots$
such that $\eps_k<1/k$, $\tilde T_k>k$,
$\|u^{(k)}(\cdot,T_k)\|>k$, 
$u^{(j)}(\cdot,0)\leq u^{(k)}(\cdot,0)+\eps_{k+1}$ if $j>k$,
$\|u(\cdot,\tilde T_k)\|<1/k$
whenever $0\leq u(\cdot,0)\leq u^{(k)}(0,\cdot)+\eps_{k+1}$,
and the solutions with initial data $u^{(k)}(\cdot,0)+v_k\in X$ blow up
in finite time for suitable $v_k$ satisfying $\|v_k\|<1/k$.
Let $u^*$ be the solution with initial data
$u^*_0:=\lim_{k\to\infty}u^{(k)}(\cdot,0)\in X$.
Since $u^*_0\leq u^{(k)}(\cdot,0)+\eps_{k+1}$ for each $k$,
the solution $u^*$ exists globally.
Since the solutions with initial data
$\max(u^*(\cdot,0),u^{(k)}(\cdot,0)+v_k)\in X$ blow up in finite time
for each $k$, Theorem~\ref{thm-strong} guarantees that
$u^*$ is a global strong threshold
solution.
Obviously, $u^*$ satisfies \eqref{osc0}.
\qed
\end{example}

%-----------------------------------------------------------------
\begin{remark} \label{rem-pS} \rm
(i) Assume $p=p_S$.
As mentioned above,
\cite[Conjecture~1.1]{FK12} suggests that
if the initial data behave like $|x|^{-\gamma}$ as
$|x|\to\infty$, then the threshold solution 
$u^*$ should grow up if $n=3$ and $\gamma>1$
(or $n=4$ and $\gamma>2$).
A rigorous result for $n=3$ has recently been obtained in \cite{dPMW16}.
More precisely, if $n=3$ and $p=p_S$ then
the results in \cite{dPMW16} guarantee the existence
of radially symmetric initial data $u_0^G\in BC^+$ such that
the corresponding solution $u^G$ of \eqref{Fuj} satisfies
$\lim_{t\to\infty}u^G(0,t)=\infty$.
Let us show that this implies the existence of
a strong threshold solution (with initial data in $X$)
which grows up.
In fact, set $M:=\|u_0^G\|_\infty$ and let $u_0^G(x)=U_0^G(|x|)$.
Taking $\delta$ small, the solution of \eqref{Fuj}
with initial data satisfying $0\leq u_0(x)\leq 3M$
if $|x|\leq2\delta$, $u_0(x)=0$ if $|x|>2\delta$,
exists globally and decays to zero, see \cite[Theorem 20.15]{SPP}.
Set $r_0:=\min\{r\in(0,2\delta]:U_0^G(r)=4M-2Mr/\delta\}$
and notice that $r_0>\delta$.
Given $r\in(r_0,2\delta)$,
set $\tilde U_0^G(r):=\min\{U_0^G(\rho):\rho\in[r_0,r]\}$.
Set also
$$
\Phi(r):=
\begin{cases}
 2M &\hbox{ if }r\in[0,\delta],\\
 4M-2Mr/\delta &\hbox{ if }r\in(\delta,r_0],\\
 \min(4M-2Mr/\delta,\tilde U_0^G(r)) &\hbox{ if }r\in(r_0,2\delta],\\
 0 &\hbox{ if }r>2\delta.
\end{cases}   
$$
Then $\Phi\in X_1$ and the solution of \eqref{Fuj}
with initial data $u(x,0)=\Phi(|x|)$ exists globally and decays to zero.
This fact and the definition of $\Phi$ also imply $z(\Phi-U_0^G)=1$.
Given $\alpha>0$, set
$\Phi_\alpha:=\Phi+\alpha\varphi$,
where $\varphi\in X_1\setminus\{0\}$ has
support in $[0,\delta]$.
Then $z(\Phi_\alpha-U_0^G)=1$ for all $\alpha$ and
the solution of \eqref{Fuj} with initial data
$u_0(x)=\Phi_\alpha(|x|)$ exists globally or blows up
in finite time 
if $\alpha$ is small or large, respectively.
Let $u^*(x,t)$ denote the corresponding (strong) threshold solution.
The global existence of threshold solutions
(for $p=p_S$ and initial data in $X$)
and the proof of \eqref{uuG0} 
guarantee 
$u^*(0,t)>u^G(0,t)$ for all $t>0$.
Since $u^G(0,t)\to\infty$ as $t\to\infty$, $u^*$ has to grow up.

(ii) Arguments in (i) suggest that if $p=p_S$ and $n=3$ (or $n=4$) 
then threshold solutions with compactly supported initial data in $X$
grow up. If this is true and if, in addition,
the corresponding sub-threshold solutions decay to zero as $t\to\infty$
(which is also plausible) 
then one can easily modify the arguments in Example~\ref{exam3}
in order to find a strong treshold solution satisfying \eqref{osc0}.
\qed
\end{remark}

%------------------------------------------------------------------
\section{Proofs of Theorem~\ref{thm-weak} for $p<p_{JL}$ and Theorem~\ref{thm-strong2}} 
\label{sec5}

%--------------------------------------------------
In this section we will prove Theorem~\ref{thm-weak} for $p<p_{JL}$
and Theorem~\ref{thm-strong2}.

Given $\ell,m>0$, set
\begin{equation} \label{Phi-exam2}
\Phi_\ell(r)=\Phi_{\ell}(r;m):=\min(m,\ell r^{-2/(p-1)})
\end{equation} 
and notice that
\begin{equation} \label{PhiInv}
\Phi_\ell(r;\lambda^{2/(p-1)}m)=\lambda^{2/(p-1)}\Phi_\ell(\lambda r;m),
\qquad \lambda>0.
\end{equation}
Fix $m>0$.
The scaling invariance of \eqref{Fuj}, \eqref{PhiInv}
and \cite[Theorem 20.6]{SPP} guarantee that
the solution $u^{(\ell)}$ with initial data
$\Phi_\ell$ exists globally. 
On the other hand, $u^{(\ell)}$ blows up
in finite time if $\ell$ is large enough,
due to scaling arguments and \cite[Theorem 17.1]{SPP}.
Set $\ell^*:=\sup\{\ell:u^{(\ell)}\hbox{ is global}\}$
and let $u^{[m]}:=u^{(\ell^*)}$.
Notice that $\ell^*>L$ if $p\geq p_S$ due to \cite{SW03}
and that $\ell^*$ is independent of $m$ due to
the scaling invariance of \eqref{Fuj} and \eqref{PhiInv}.
If $p<p_S$ then the universal bounds in \cite{PQS2}
and continuous dependence on initial data guarantees
that $u^{[m]}$ is global and decays to zero as $t\to\infty$.
If $p\in[p_S,p_{JL})$ then choosing $\ell_1\in(L,\ell^*)$
and $\beta>m$ large enough
we have $U_\beta(r)<\ell_1r^{-2/(p-1)}$ for all
$r\geq r_1$, where $r_1$ is defined by $m=\ell_1r_1^{-2/(p-1)}$.
Consequently, $z(\Phi_\ell-U_\beta)=1$ for all $\ell$ close to $\ell^*$,
and the same arguments as in Example~\ref{exam-decay}
guarantee that $u^{[m]}$ is global and decays to zero as $t\to\infty$.

Since $\Phi_{\ell^*}(\cdot;m_1)\leq\Phi_{\ell^*}(\cdot;m_2)$
if $m_1\leq m_2$, $u^{[m]}$ is a weak threshold solution for any $m$,
and $u^{[m]}\nearrow \tilde u$ as $m\to\infty$,
where $\tilde u$ is the minimal weak solution of \eqref{Fuj}
with initial data $\tilde u_0(x)=\ell^*|x|^{-2/(p-1)}$. 
Due to the scaling invariance, given $\lambda>0$, we have
$u^{[\Lambda m]}(x,t)=\Lambda u^{[m]}(\lambda x,\lambda^2t)$,
where $\Lambda:=\lambda^{2/(p-1)}$.
Passing to the limit as $m\to\infty$ we see that $\tilde u$ is self-similar.
In addition, the global existence of $u^{[m]}$
guarantees a uniform bound for $u^{[m]}(\cdot,t)$ in $L^1_{\loc}(\R^n)$
and $u^{[m]}$ in $L^p_{\loc}(\R^n\times[0,\infty))$,
hence the same bounds are true for $\tilde u$.
Since $u^{[m]}$, hence $\tilde u$ are radially nonincreasing,
these bounds guarantee that $\tilde u$ is a global classical
solution in $(\R^n\setminus\{0\})\times(0,\infty)$.
Consequently, $\tilde u(x,t)=t^{-1/(p-1)}w(|x|/\sqrt{t})$,
where $w$ is a classical positive solution of the equation
\begin{equation} \label{eqw}
 w''+\bigl(\frac{n-1}r+\frac r2\Bigr)w'+\frac1{p-1}w+w^p=0
\qquad \hbox{for }r>0.
\end{equation}

If $p<p_S$ then the universal bounds in \cite{PQS2}
guarantee that $w$ is bounded.
Let $p\geq p_S$.
Given $x\ne0$, we have $\lim_{t\to0}u^{[m]}(x,t)=\ell^*|x|^{-2/(p-1)}$
for $m$ large enough. Since $\tilde u\geq u^{[m]}$, we have 
$\liminf_{t\to0}\tilde u(x,t)\geq \ell^*|x|^{-2/(p-1)}$,
hence 
\begin{equation}\label{winfty}
\liminf_{r\to\infty}w(r)r^{2/(p-1)}\geq\ell^*>\ell_1>L.
\end{equation}
In addition, since $u^{[m]}(x,t)=U^{[m]}(|x|,t)$
and $z(U^{[m]}(\cdot,t)-U_*)\leq1$, we also
have
\begin{equation} \label{zw}
z(w-U_*)\leq 1.
\end{equation}
The uniqueness result for singular solutions of \eqref{eqw} in \cite{Quniq}, 
\eqref{winfty} and \eqref{zw} guarantee that $w$ is bounded.
The boundedness of $w$ (for both $p<p_S$ and $p\geq p_S$)
guarantees that $w\in C^2([0,\infty))$ and $w$ satisfies
the initial condition $w'(0)=0$ (see \cite{Quniq}). 
In addition, $w$ is nonincreasing.
Since 
$\ell(w):=\lim_{r\to\infty}r^{2/(p-1)}w(r)$  is well defined
for any bounded solution $w$ of \eqref{eqw}) due to \cite{HW82}, 
\eqref{winfty} implies $\ell(w)\geq\ell^*$.

It is known (see \cite{Nai12} and the references therein)
that there exists a unique bounded solution $w^*$ of \eqref{eqw}
for which the value $\ell(w^*)$ is maximal. 
Then  
$$v(x,t):=(t+1)^{-1/(p-1)}w^*(|x|/\sqrt{t+1})$$
is a global solution of \eqref{Fuj} with positive
initial data $v_0(x)=w^*(|x|)$ so that the choice of
$\ell^*$ guarantees
\begin{equation}\label{alphaell}
\ell^*\geq\ell(w^*).
\end{equation}
Now $\ell^*\leq\ell(w)\leq\ell(w^*)$ 
and \eqref{alphaell} imply $\ell(w^*)=\ell^*$,
hence $w=w^*$.

Let $u_0\in BC^+$ and let
$u$ denote the solution of \eqref{Fuj}.
First assume  $u_0(x)|x|^{2/(p-1)}\leq\ell^*$.
Choose $m>0$ such that 
$u_0(x)\leq\Phi_{\ell^*}(|x|;m)$.
Then 
$$u(x,t)\leq u^{[m]}(x,t)\leq\tilde u(x,t)
 =t^{-1/(p-1)}w(|x|/\sqrt{t}),\quad x\in\R^n,\ t>0,$$
hence $u$ exists globally and 
$\|u(\cdot,t)\|_\infty\leq Ct^{-1/(p-1)}$.

Next assume $A:=\liminf_{|x|\to\infty}u_0(x)|x|^{2/(p-1)}>\ell^*$.
Then there exists $R>0$ such that $u_0(x)|x|^{2/(p-1)}>(A+\ell^*)/2$
for $|x|\geq R$.
Set $\hat u_0(x)=\hat\Phi(|x|)$, where
$$ 
\hat\Phi(r):= \begin{cases} 
0 &\hbox{ for $r\leq R$}, \\
\frac12 (A+\ell^*) r^{-2/(p-1)} &\hbox{ for $r\geq R+1$}, \\
(r-R)\hat\Phi(R+1) &\hbox{ for $r\in(R,R+1)$},
\end{cases}
$$
and let $\hat u(x,t)=\hat U(|x|,t)$ be the solution of \eqref{Fuj}
with initial data $\hat u_0$.
Then $0\leq\hat u_0\leq u_0$ and it is sufficient to prove that
$\hat u$ blows up in finite time.
Fix $t_0>0$ small and denote $c_0:=\hat u(0,t_0)>0$.
Since $\ell:=(A+3\ell^*)/4>\ell^*$,
the solutions $v^{[m]}(x,t)=V^{[m]}(|x|,t)$ with initial data $\Phi_{\ell}(|x|;m)$
blow up in finite time for any $m>0$.
Notice that  $V^{[m]}(0,0)=m>0=\hat U(0,0)$ and
$z(V^{[m]}(\cdot,0)-\hat U(\cdot,0))=1$.
In addition, the continuous dependence of solutions on initial data 
in $L^\infty$ guarantees $V^{[m]}(0,t_0)<c_0=\hat U(0,t_0)$ for $m$
small enough, hence $z(V^{[m]}(\cdot,t_0)-\hat U(\cdot,t_0))=0$,
$V^{[m]}(\cdot,t_0)\leq \hat U(\cdot,t_0)$.
Since $v^{[m]}$ blows up in finite time and 
$\hat u(\cdot,t)\geq v^{[m]}(\cdot,t)$ for $t\geq t_0$, 
the solution $\hat u$ blows up in finite time as well. 

It remains to prove that
the solution of \eqref{Fuj} with initial data $\tilde u(\cdot,t_0)$
is a strong threshold solution for any $t_0>0$.
Without loss of generality, we can assume $t_0=1$.
In fact, once we prove the assertion for $t_0=1$,
the assertion for $t_0>1$ follows from Lemma~\ref{lem-t0}
and Theorem~\ref{thm2},
and the assertion for $t_0\in(0,1)$ follows from the comparison principle.

Set $u_0(x):=\tilde u(x,1)=w(|x|)$ and notice that $u_0\in X$.
Since the corresponding solution $u(x,t)=\tilde u(x,t+t_0)$
exists globally, fixing $\eps>0$ is it sufficient to find
$\varphi\in X_1$ such that $\|\varphi(|\cdot|)\|<\eps$ and 
the solution $\hat u$ of \eqref{Fuj} with initial data $u_0(x)+\varphi(|x|)$ 
blows up in finite time, 
see Theorem~\ref{thm-strong}.

Let $\hat w$ be defined by
$$ \hat u(x,t)=(t+1)^{-1/(p-1)}\hat w(r,s), \qquad 
      r=\frac{|x|}{\sqrt{t+1}}, \ s=\log(t+1). $$
Then $\hat u(x,0)=\hat w(|x|,0)$, ${\hat w}_r(0,s)=0$, and
\begin{equation}  \label{eqz}
 {\hat w}_s={\hat w}_{rr}+\bigl(\frac{n-1}r+\frac r2\Bigr){\hat w}_r
             +\frac1{p-1}{\hat w}+{\hat w}^p,  \quad r,s>0. 
\end{equation}
Notice that ${\hat w}_0(r):={\hat w}(r,0)$ satisfies ${\hat w}_0=w+\varphi$,
and $w$ is a stationary solution of \eqref{eqz}.

\cite[Lemma 3.2(i), Remark 3.7 and the proof of Lemma~D.1]{Nai06}
guarantee that there exists a positive function $\phi\in C^2([0,\infty))$ such that
$\phi(|\cdot|)\in L^1(\R^n)$,  $\phi'(0)=0$ and   
$$ \phi_{rr}+\bigl(\frac{n-1}r+\frac r2\Bigr)\phi_r+\frac1{p-1}\phi+pw^{p-1}\phi=0,\quad r>0.$$
It is easy to see that the solution of this equation has to be nonincreasing, 
hence $\phi\in X_1$.
Taking a suitable positive multiple of $\phi$ we may assume $\|\phi(|\cdot|)\|<1$.
Set  $\varphi:=\eps \phi$.

Then $(\hat w)_s(\cdot,0)=(w+\varphi)^p-w^p-pw^{p-1}\varphi>0$,
hence $\hat w$ is a time-increasing solution.
Assume on the contrary that $\hat w$ (hence also $\hat u$) is global.  
If $p<p_S$ then the universal estimate in \cite{PQS2} for $\hat u$ 
guarantee that $\hat w$ is bounded,
hence ${\hat w}(\cdot,t)$ has to converge to a bounded steady state lying above
$w$, which contradicts the nonexistence of such steady states. 
Next assume $p\geq p_S$
and notice that $z({\hat w}_0-U_*)=1$ if $p=p_S$ and $\eps$ is small enough.
Since \eqref{eqz} possesses neither bounded nor unbounded 
steady states above $w$ due to \cite{Quniq},
there exists $r_0>0$ such that
${\hat w}(r_0,t)\to\infty$ as $t\to\infty$.
Since ${\hat w}_r\leq0$, 
we have $\inf_{0<r<r_0}{\hat w}(r,t)={\hat w}(r_0,t)\to\infty$ as $t\to\infty$,
and comparison with the corresponding boundary value problem
in $(0,r_0)$ (with boundary conditions $w_r(0,t)=w(r_0,t)=0$)
and a Kaplan-type argument easily guarantee blow-up of $\hat w$,
which concludes the proof.

%--------------------------------------------
\vskip3mm
{\bf Acknowledgements.}
The author was supported in part by the Slovak Research and Development Agency 
under the contract No. APVV-14-0378 and by VEGA grant 1/0319/15.

%----------------------------
\def\by{\relax}    
\def\paper{\relax} 
\def\jour{\textit} 
\def\vol{\textbf}  
\def\yr#1{\rm(#1)} 
\def\pages{\relax} 
\def\book{\textit} 
\def\inbook{In: \textit}
\def\finalinfo{\rm} 
%----------------------------

\bibliographystyle{amsplain}

\begin{thebibliography}{16}
%
\bibitem{BS15}
\by{S. Blatt and M. Struwe}, 
\paper{An analytic framework for the supercritical Lane-Emden
            equation and its gradient flow}, 
\jour{Int. Math. Res. Notices}
\vol{2015} \yr{2015}, \pages{2342--2385}.
%
\bibitem{CDZ07}
\by{K.-S. Chou, S.-Z. Du and G.-F. Zheng},
\paper{On partial regularity of the borderline solution of semilinear
            parabolic problems},
\jour{Calc. Var. Partial Differential Equations}
\vol{30} \yr{2007}, \pages{251--275}.
%
\bibitem{CMR}
\by{C. Collot, F. Merle and P. Rapha\"el},
\paper{Dynamics near the ground state for the energy critical nonlinear heat
       equation in large dimensions},
\finalinfo{Preprint arXiv:1604.08323}.
%
\bibitem{dPMW16}
\by{M. del Pino, M. Musso and J. Wei},
\paper{Infinite time blow-up for critical heat equation in $\R^3$
with fast decay initial condition},
\finalinfo{In preparation}.
%
\bibitem{FK12}
\by{M. Fila and J.R. King},
\paper{Grow up and slow decay in the critical Sobolev case},
\jour{Netw. Heterog. Media}
\vol{7}  \yr{2012}, \pages{661--671}.
%
\bibitem{GV97}
\by{V.A. Galaktionov and J.L. V\'azquez},
\paper{Continuation of blow-up solutions of
      nonlinear heat equations in several space dimensions},
\jour{Comm. Pure Appl. Math.}
\vol{50}  \yr{1997},  \pages{1--67}.
%
\bibitem{GNW92}
\by{C. Gui, W.-M. Ni and X. Wang},
\paper{On the stability and instability of positive
       steady states of a semilinear heat equation in ${\mathbb R}^n$},
\jour{Comm. Pure Appl. Math.}
\vol{45}  \yr{1992}, \pages{1153--1181}.
%
\bibitem{GNW01}
\by{C. Gui, W.-M. Ni and X. Wang},
\paper{Further study on a nonlinear heat equation},
\jour{J. Differential Equations}
\vol{169} \yr{2001}, \pages{588--613}.
%
\bibitem{HW82}
\by{A. Haraux and F.B. Weissler},
\paper{Non-uniqueness for a semilinear initial value problem},
\jour{Indiana Univ. Math. J.}
\vol{31} \yr{1982}, \pages{167--189}.
%
\bibitem{Kaw96}
\by{T. Kawanago},
\paper{Asymptotic behavior of solutions of a semilinear heat equation
            with subcritical nonlinearity},
\jour{Ann. Inst. H. Poincar\'e Anal. non lin\'eaire}
\vol{13}  \yr{1996}, \pages{1--15}.
%
\bibitem{MM04}
\by{H. Matano and F. Merle},
\paper{On nonexistence of type II blowup for a
          supercritical nonlinear heat equation},
\jour{Comm. Pure Appl. Math.}
\vol{57}  \yr{2004}, \pages{1494--1541}.
%
\bibitem{MM09}
\by{H. Matano and F. Merle},
\paper{Classification of type I and type II behaviors for a
            supercritical nonlinear heat equation},
\jour{J. Funct. Anal.}
\vol{256}  \yr{2009},  \pages{992--1064}.
%
\bibitem{MM11}
\by{H. Matano and F. Merle},
\paper{Threshold and generic type I behaviors for a supercritical
       nonlinear heat equation},
\jour{J. Funct. Anal.}
\vol{261}  \yr{2011},  \pages{716--748}.
%
\bibitem{Mi02}
\by{N. Mizoguchi},
\paper{On the behavior of solutions for a semilinear parabolic equation
       with supercritical nonlinearity},
\jour{Math. Z.}
\vol{239} \yr{2002}, \pages{215--229}.
%
\bibitem{Mi05}
\by{N. Mizoguchi},      
\paper{Boundedness of global solutions for a supercritical semilinear heat
            equation and its application},
\jour{Indiana Univ. Math. J.}
\vol{54} \yr{2005}, \pages{1047--1059}.
%
\bibitem{Nai06}
\by{Y. Naito},
\paper{An ODE approach to the multiplicity of self-similar solutions
            for semi-linear heat equations},
\jour{Proc. Roy. Soc. Edinburgh Sect. A}
\vol{136} \yr{2006},  \pages{807--835}.
%
\bibitem{Nai12}
\by{Y. Naito},
\paper{The role of  forward self-similar 
solutions in the Cauchy problem for semilinear heat equations},
\jour{J. Differential Equations} 
\vol{253} \yr{2012}, \pages{3029--3060}.
%
\bibitem{NN85}
\by{W.-M. Ni and R. Nussbaum},
\paper{Uniqueness and nonuniqueness for positive radial solutions
             of $\Delta u+f(u,r)=0$},
\jour{Comm. Pure Appl. Math.}
\vol{38} \yr{1985}, \pages{67--108}.
%
\bibitem{PQS2}
\by{P. Pol\'a\v{c}ik and P. Quittner and Ph. Souplet},
\paper{Singularity and decay estimates in superlinear problems
       via Liouville-type theorems. Part II: Parabolic equations},
\jour{Indiana Univ. Math. J.} 
\vol{56} \yr{2007}, \pages{879--908}.
%
\bibitem{PY03}
\by{P. Pol\'a\v cik and E. Yanagida},
\paper{On bounded and unbounded global solutions of a supercritical
             semilinear heat equation},
\jour{Math. Ann.}
\vol{327} \yr{2003}, \pages{745--771}.
%
\bibitem{PY14}
\by{P. Pol\'a\v cik and E. Yanagida},
\paper{Global unbounded solutions of the Fujita equation in the intermediate range},
\jour{Math. Ann.} 
\vol{360} \yr{2014},  \pages{255--266}.
%
\bibitem{PY14s}
\by{P. Pol\'a\v cik and E. Yanagida},
\paper{Localized solutions of a semilinear parabolic equation
            with a recurrent nonstationary asymptotics},
\jour{SIAM J. Math. Anal.}
\vol{46} \yr{2014}, \pages{3481--3496}.
%
\bibitem{Q03}
\by{P. Quittner},
\paper{Continuity of the blow-up time and a~priori bounds
             for solutions in superlinear parabolic problems},
\jour{Houston J. Math.}
\vol{29} \yr{2003}, \pages{757--799}.
%
\bibitem{Q16}
\by{P. Quittner},
\paper{Liouville theorems for scaling invariant superlinear parabolic problems
            with gradient structure},
\jour{Math. Ann.}
\vol{364} \yr{2016},
\pages{269--292}.

\bibitem{Quniq}
\by{P. Quittner},
\paper{Uniqueness of singular 
forward or backward self-similar solutions of a semilinear heat equation},
\finalinfo{Preprint}.
%
\bibitem{SPP}
\by{P. Quittner and Ph. Souplet},
\book{Superlinear parabolic problems. Blow-up, global existence and steady
states},
Birkh\"auser Advanced Texts, Birkh\"auser, Basel, 2007. 
%
\bibitem{Sou16}
\by{Ph. Souplet},
\paper{Morrey spaces and classification of global solutions
       for a supercritical semilinear heat equation in $\R^n$},
\finalinfo{Preprint arXiv:1604.01667}.
%
\bibitem{SW03}
\by{Ph. Souplet and F.B. Weissler},
\paper{Regular self-similar solutions of the nonlinear heat equation with initial data
above the singular steady state},
\jour{Ann. Inst. H. Poincar\'e Anal. Non Lin\'eaire}
\vol{20} \yr{2003}, \pages{213--235}.
%

\end{thebibliography}

\end{document}